\documentclass[11pt]{amsart}
\usepackage{amssymb,amsmath}

\newcommand{\nc}{\newcommand}

\nc{\BP}{{\mathbb P}}
\nc{\cO}{{\mathcal O}}
\nc{\cK}{{\mathcal K}}
\nc{\bD}{{\mathbb D}}
\nc{\C}{{\mathbb C}}
\nc{\Z}{{\mathbb Z}}
\nc{\spn}{\operatorname{span}}
\nc{\vareps}{\varepsilon}
\nc{\asc}{\operatorname{asc}}
\nc{\wt}{\operatorname{wt}}
\nc{\htt}{\operatorname{ht}}
\nc{\qdim}{\operatorname{qdim}}
\nc{\Cx}{{\C^{\times}}}
\nc{\cP}{{\mathcal P}}
\nc{\Gr}{{\mathsf{Gr}}}

\newtheorem{Cor}{Corollary}[section]
\newtheorem{Lem}[Cor]{Lemma}
\newtheorem{Prop}[Cor]{Proposition}
\newtheorem{Theorem}[Cor]{Theorem}

\theoremstyle{definition}
\newtheorem{Remark}[Cor]{Remark}

\begin{document}
\title{Fixed points of reverse Hessenberg convolution varieties}
\author{Joel Kamnitzer}
\email{jkamnitz@gmail.com}
\address{Department of Mathematics and Statistics \\ McGill University}

\begin{abstract}
Associated to any unit interval graph, Syu Kato introduced a variety which gives (via the geometric Satake correspondence) a graded $GL_m$ representation whose character is the chromatic quasisymmetric polynomial of the graph.  In this short note, we reprove Kato's result by analyzing the fixed points of his varieties.
\end{abstract}

\maketitle
\section{Introduction}
Let $ n $ be a natural number.  As usual we write $ [n] = \{1, \dots, n\} $.  We will study the q-analog of the chromatic symmetric polynomial for unit interval graphs on the vertex set $[n]$.  Such graphs are in bijection with various combinatorial objects, most notably Dyck paths, but for the purposes of this note, we introduce the following.

A \textbf{reverse Hessenberg function} is a weakly increasing map $ r : [n] \rightarrow \{0, \dots, n\} $ satisfying $ r(i) < i  $ for all $i$.    Associated to the reverse Hessenberg function $ r $, there is a unit interval graph $ \Gamma_r$.  The vertex set of this graph is $ [n] $ and two points $ j < i $ are connected by an edge if $ r(i) < j $.  For example, if $ r(i) = i-1 $ for all $ i $, then there are no edges in the graphs (we call this the \textbf{staircase function}).  On the other hand, if $ r(i) = 0 $ for all $ i $, then we have the complete graph.   Let $ E_r = \sum_i i - 1- r(i) $.  This is the total number of edges in the graph $ \Gamma_r $.

Let $ m $ be another natural number.  A \textbf{proper colouring} of $\Gamma_r $ is a function $ \kappa : [n] \rightarrow [m] $ such that if $ r(i) < j < i $ then $ \kappa(j) \ne \kappa(i) $.  Let $ C_r $ denote the set of proper colourings of $ \Gamma_r $.  For $ C_r $ to be non-empty, we must have $ i - r(i) \le m $ for all $ i $ and so from now on, we will assume that this condition holds.

Let  $ \kappa \in C_r $ be a proper colouring.  The \textbf{weight} of $ \kappa $,  $ \wt(\kappa) =  \vareps_{\kappa(1)} + \cdots + \vareps_{\kappa(n)} \in \Z^m$ is the integer vector which records how many of each colour was used.  As \textbf{ascent} of $ \kappa $ is a pair $ j, i \in [n] $ such that $ r(i) < j < i $ and $ \kappa(j) < \kappa(i) $.  We write $ \asc(\kappa,i) $ for the number of ascents ending at $ i $, and $ \asc(\kappa) $ for the total number of ascents.  Finally, the \textbf{height} of $ \kappa $, denoted $ \htt(\kappa) $,  is defined as $ \kappa(1) + \dots + \kappa(n) $.

Following Shareshian-Wachs \cite{SW}, we consider the chromatic quasisymmetric polynomial in $ m $ variables, which is defined as
$$
CSP_r = \sum_{\kappa \in C_r} q^{\asc(\kappa)} x^{\wt(\kappa)} \in \Z[q][x_1, \dots, x_m]^{S_m}
$$
Note that this is a truncation of their symmetric function to $ m $ variables.

Since $ \Z[x_1, \dots, x_m]^{S_m} $ is the space of characters of polynomial $ GL_m $ representations, it is natural to look for a graded $GL_m $ representation whose character is $ CSP_r $.  Kato \cite{Kato} solved this problem by constructing certain varieties which map to the affine Grassmannian and using the geometric Satake correspondence.  In this note, we will give a new proof of Kato's result by examining fixed points in his varieties.

\section{Kato's varieties}

\subsection{Reverse Hessenberg convolution space}
Let $ \cO = \C[[t]],\ \cK = \C((t)), \ G = GL_m$.  Consider the affine Grassmannian $ \Gr = G(\cK)/G(\cO)$ of $ GL_m $, which will identify with the set of $ \cO $ lattices in $ \cK^m $.  Following Kato \cite{Kato}, we define the \textbf{reverse Hessenberg convolution space} to be
$$
Y_r := \{ \cO^m = L_0 \subset L_1 \subset \cdots \subset L_n \ : \ L_i \in \Gr, \dim L_{i+1}/L_i = 1, t L_i \subset L_{r(i)} \}
$$

If $ r $ is the staircase function, then $ Y_r $ is the convolution variety $ \Gr^{\omega_1} \tilde{\times} \cdots \tilde{\times} \Gr^{\omega_1} $.   On the other extreme, if $ r(i) = 0 $ for all $ i$ (which forces $ n \le m $), then $ Y_r $ is the partial flag variety in $ \C^m $ of subspaces of dimensions $ (1, \dots, n, m) $.

Kato's varieties are iterated projective space bundles.  Given $ r $ as above, let $ \hat r $ denote the restriction of $ r $ to $ [n-1] $.  This is still a reverse Hessenberg function and we have the variety $ Y_{\hat r}$.

\begin{Prop} \label{pr:bundle}
The natural map $ Y_r \rightarrow Y_{\hat r} $ given by $ (L_0, \dots, L_n) \mapsto (L_0, \dots, L_{n-1}) $ is a bundle of projective spaces.  More precisely the fibre is $ \BP(t^{-1} L_{r(n)} / L_{n-1}) $ which is a projective space of dimension $ m - n + r(n) $.
\end{Prop}

\begin{proof}
The condition on $ L_n $ shows that $ L_{n-1} \subset L \subset t^{-1} L_{r(n)}$ and hence $ L / L_{n-1} $ is a 1-dimensional subspace of the vector space $ t^{-1}L_{r(n)} / L_{n-1} $ which is of dimension $ m - (n -1 - r(n))$.  Conversely, any 1-dimensional subspace of this vector space determines a lattice $ L_n $.
\end{proof}

\begin{Cor} \label{cor:dim}
The space $ Y_r $ is smooth and has dimension $ d_r := (m-1)n - E_r $.
\end{Cor}

\begin{Remark}
When $ n = m $, then $ Y_r $ is the compactification (via the well-known embedding of the nilpotent cone into the affine Grassmannian) of the following vector bundles on the flag variety:
$$
Y_r^\circ := \{ (X, V_\bullet) \in Fl_n \times \mathcal N : X V_i \subset V_{r(i)} \}
$$
In turn, this vector bundle is the orthogonal complement of the usual family of Hessenberg varieties, which was studied by Brosnan-Chow \cite{BC} and others.  The variety $ Y_r^\circ $ can be used to produce an $ S_n $ representation using Springer theory.  From our perspective, the main advantages of Kato's varieties $ Y_r $ are that they can be used to produce $ GL_m $ representations and that they have many torus fixed points, which we study below.
\end{Remark}

\subsection{Fixed points and colourings}

We have an action of $ G(\cO) $ on $ Y_r $ and in particular, we can consider the torus $ T \subset G(\cO) $.  Recall that the $ T $-fixed points in $ \Gr $ are indexed by the coweight lattice $ \Z^m $.  For our purposes, we will choose the convention that $ L_\mu := \spn_{\cO}(t^{-\mu_1} e_1, \dots, t^{-\mu_n} e_n) $ for $ \mu \in \Z^m $.

Thus each fixed point in $ Y_r $ is given by a sequence $ L_0 \subset L_{\mu(1)} \subset \cdots \subset L_{\mu(m)} $ where $ \mu(i) \in \Z^m $ for each $ i$.  The condition that $ L_{\mu({i-1})} \subset L_{\mu(i)} $ with 1-dimensional quotient forces $ \mu(i)  = \mu({i-1}) + \vareps_{\kappa(i)} $ for some $ \kappa(i) \in [m] $ (here $ i = 1, \dots, n$ and $ \mu(0) = 0 $).  Thus a $T$-fixed point is encoded by a function $ \kappa : [n] \rightarrow [m] $, which we think of a colouring of the graph $ \Gamma_r $.

\begin{Lem}
Given a proper colouring $ \kappa $ of $ \Gamma_r $, let $ \mu(i) = \vareps_{\kappa(1)} + \cdots + \vareps_{\kappa(i)} $ for $ i = 1, \dots, n$.
This defines a bijection between $C_r$ and $ Y_r^T $.
\end{Lem}

\begin{proof}
Given a point $ L_0 \subset L_{\mu(1)} \subset \cdots \subset L_{\mu(m)} $ as above, we note that $ t L_i \subset L_{r(i)} $ if and only if for all $p \in [m] $, we have $ \mu(i)_p \le \mu(r(i))_p + 1$.  This is equivalent to the condition that $ \kappa(j) \ne \kappa(j') $ for all $ j, j' \in \{r(i) +1, \dots, i \}$.  Thus $ \kappa $ gives a point in $ Y_r $ if and only if it is a proper colouring.
\end{proof}

Given $ \kappa \in C_r $, we will simply write $ \kappa :=(L_0, L_{\mu(1)}, \dots, L_{\mu(n)}) $, where $ \mu(i) =  \vareps_{\kappa(1)} + \cdots + \vareps_{\kappa(i)} $, for the resulting fixed point in $ Y_r $.  For any such $ \kappa $, we can consider the Bialynicki-Birula cell  $ A_\kappa = \{ y \in Y_r : \lim_{s \to \infty} s y = \kappa \} $, where we have fixed $ \Cx \subset T $ by $ s \mapsto (s, s^2, \dots, s^n) $.  Since $ Y_r $ is smooth and projective, each $ A_\kappa $ is an affine space and these give an affine paving of $ Y_r$.

\begin{Lem} \label{le:asc}
$A_\kappa = \C^d $ where $ d = \htt(\kappa) - \asc(\kappa) - n$
\end{Lem}

\begin{proof}
We consider the action of $ \Cx $ on the tangent space $ T_\kappa Y_r $.  Let $T^-_\kappa Y_r $ be the direct sum of negative weight spaces for the action of $ \Cx $.  By the general theory, we have $ d = \dim T^-_\kappa Y_r$.

We calculate $ \dim T^-_{\kappa} Y_r $ recursively, exploiting the projective bundle structure from Proposition \ref{pr:bundle}.  Let $ \hat r, \hat \kappa $ denote the restrictions to $[n-1] $.  We see that $$ \dim T^-_\kappa Y_r = \dim T^-_{\hat \kappa} Y_{\hat r}  + \dim T^-_{L_{\mu(n)}} \BP(t^{-1} L_{\mu(r(n))} / L_{\mu(n-1)}))$$
Now the tangent space to the projective space is given by $$ (L_{\mu(n)} / L_{\mu({n-1})})^* \otimes (t^{-1}L_{\mu(r(n))} / L_{\mu(n)})$$ which has weights $ - \kappa(n) + p $ where $ p \in [m] \setminus \{ \kappa(r(n) +1), \dots, \kappa(n) \} $.  In particular the dimension of the negative part of this tangent space is the number of $ p \in [m] \setminus \{ \kappa(r(n) +1), \dots, \kappa(n) \} $ such that $ p < \kappa(n) $.  Thus the dimension of this negative part is $ \kappa(n) - 1 - \asc(\kappa,n) $ and so the result follow.
\end{proof}

\subsection{Analysis of push forward}

We have a natural morphism $ m : Y_r \rightarrow \Gr $ given by $ (L_0, \dots, L_n) \mapsto L_n $.  This map is not in general semismall.  The decomposition theorem still tells us that the image of the constant sheaf is a direct of shifted semisimple perverse sheaves.  Since $ G(\cO) $ acts on $ Y_r $ compatibly with its action of $ \Gr$, these perverse sheaves must all lie in the Satake category $ \operatorname{Perv}_{G(\cO)}(\Gr) $ which is a semisimple category with simple objects $ IC(\lambda) $, the IC sheaf of $ \overline \Gr^\lambda $.  Hence we find that
\begin{equation} \label{eq:push}
m_* \C_{Y_r} [d_r] = \oplus_\lambda IC(\lambda) \otimes M(\lambda)
\end{equation}
where $ M(\lambda) $ is a $ \Z$-graded multiplicity vector space, where the grading reflects the cohomological shifts.  Since $ Y_r $ is smooth and $ m $ is proper, the right hand side is invariant under Verdier duality and so each multiplicity space $ M(\lambda) $ is palindromic: $ \dim M(\lambda)_i = \dim M(\lambda)_{-i}$.    Given a $\Z$-graded vector space $ V $, we write $ \qdim V = \sum_{i \in \Z} \dim V_{i}  q^{i/2} \in \Z[q^{1/2},q^{-1/2}] $.

Under the geometric Satake correspondence $ m_* \C_{Y_r}[d_r] $ corresponds to $ \oplus_\lambda V(\lambda) \otimes M(\lambda) $.  Since each $ M(\lambda) $ is $ \Z$-graded, we can regard the RHS as a $ \Z$-graded representation of $ GL_m $, and thus its graded character is $ \sum_\lambda \qdim M(\lambda) s_\lambda $.  Here $ V(\lambda) $ is the polynomial representation of $ GL_m $ with highest weight $ \lambda $,  $ s_\lambda \in \C[x_1, \dots, x_m]^{S_m} $ is its character, the Schur polynomial.

The following result is our main theorem, which was proven by Kato \cite[Theorem A]{Kato} using the modular law for chromatic symmetric functions.

\begin{Theorem} \label{th:main}
We have an equality $ CSP_r = q^{\frac{E_r}{2}} \sum_\lambda \qdim M(\lambda) s_\lambda $.
\end{Theorem}

In order to prove this theorem, it will be necessary to recall the weight functors of Mirkovic-Vilonen \cite{MV}.  Let $ \mu \in \Z^m$ and consider the semiinfinite orbit $$ S^\mu  = \{ L \in \Gr : \lim_{s \to \infty} L = L_\mu \} = N(\cK) L_\mu$$  Let $ j_\mu : S^\mu \rightarrow \Gr $ be the inclusion and let $ 2\rho = (m-1, \dots, 1-m) $.  The following result is an important part of the geometric Satake correspondence.

\begin{Theorem}
For any $ \cP \in \operatorname{Perv}_{G(\cO)}(\Gr) $, $ H^\bullet(S^\mu, j_\mu^!(\cP))$ is concentrated in degree $ 2\rho(\mu) $.  Moreover, if $ \cP $ corresponds to $ V \in \operatorname{Rep}(GL_m) $ under the geometric Satake correspondence, then we have an identification $ H^{2\rho(\mu)}(S^\mu, j_\mu^!(\cP)) = V_\mu $.
\end{Theorem}

\begin{Remark}
The attentive reader will have noticed that $ L_\mu $ corresponds to the usual point $ t^{-\mu} $ of the affine Grassmannian and thus $ S^\mu $ should really be denoted $ S^{-\mu} $.  If desired, we can return to the standard conventions by changing the definition of $ Y_r $ to a set of decreasing flags of lattices $ L_0 \supset \cdots \supset L_n $, but we find this alternate version less intuitive.
\end{Remark}

\begin{proof}[Proof of Theorem \ref{th:main}]
Applying $ H^{2\rho(\mu) + \bullet}(S^\mu, j_\mu^!(-)) $ to both sides of (\ref{eq:push}) and applying Theorem \ref{th:main} to the RHS yields
\begin{equation} \label{eq:push2}
H^{2\rho(\mu) + \bullet}(S^\mu, j_\mu^! m_* \C_{Y_r}[d_r]) = \oplus_\lambda V(\lambda)_\mu \otimes M(\lambda)
\end{equation}
To compute the left hand side we consider $ m^{-1} (S^\mu) \subset Y_r $ and we use that $ \C_{Y_r}[d_r] = \bD_{Y_r}[-d_r] $.  Then base change gives
$$ H^{2\rho(\mu) + \bullet}(S^\mu, j_\mu^! m_* \C_{Y_r}[d_r]) = H^{2\rho(\mu) + \bullet}(m^{-1} (S^\mu), \bD_{m^{-1}(S^\mu)}[-d_r]) = H_{ - 2\rho(\mu)+d_r - \bullet}(m^{-1}(S^\mu)) $$
where on the right we have the Borel-Moore homology.

Next since $ m $ is $ \Cx $-equivariant, we have that
$$ m^{-1} (S^\mu) = \{ y \in Y_r : m(\lim_{s \to \infty} s_y) = L_\mu \} = \bigcup_{\substack{\kappa \in C_r \\ wt(\kappa) = \mu}} A_\kappa $$
and thus $ H_{2p}(m^{-1} (S^\mu)) = | \{\kappa \in C_r : \dim A_\kappa = p , \wt(\kappa) = \mu \} |.$

Combining together the last few equations and Lemma \ref{le:asc}, we find that
$$
\qdim H^{2\rho(\mu) + \bullet}(S^\mu, j_\mu^! m_* \C_{Y_r}[d_r]) = \sum_{\substack{\kappa \in C_r \\ \wt(\kappa) = \mu}} q^{\asc(\kappa) - \htt(\kappa) +n - \rho(\mu) + d_r/2}
$$

Finally, we note that
$$ \rho(\mu) = \frac{1}{2} \sum_k m - 2c(k) +1 = \frac{n(m+1)}{2} - \htt(\kappa) $$
and hence via Corollary \ref{cor:dim},
$
\asc(\kappa) - \htt(\kappa) +n - \rho(\mu) + d_r/2 = \asc(\kappa) - E_r/2.
$
So,
$$
\qdim H^{2\rho(\mu) + \bullet}(S^\mu, j_\mu^! m_* \C_{Y_r}[d_r]) = q^{-\frac{E_r}{2}} \sum_{\kappa : \wt(\kappa) =\mu} q^{\asc(\kappa)}
$$

Thus taking graded dimension of both sides of (\ref{eq:push2}) yields
$$
q^{-\frac{E_r}{2}} \sum_{\kappa \in C_r} q^{\asc(\kappa)} =  \sum_\lambda \qdim M(\lambda) \dim V(\lambda)_\mu
$$
which implies the desired statement.
\end{proof}

\subsection*{Acknowledgements}
I thank Pierre Baumann, Patrick Brosnan, Mathieu Guay-Paquet, Syu Kato, Antoine Labelle, and Alejandro Morales for helpful conversations.

\end{document}